\documentclass[12pt,letterpaper]{amsart}
\usepackage{amssymb,setspace,mathrsfs}
\usepackage{mathpazo} %Fonts
\usepackage{geometry}
\usepackage[colorlinks,allcolors=blue]{hyperref}%pagebackref=true
\usepackage{xpatch} 
\xpatchcmd{\proof} 
{\itshape} 
{\bfseries} 
{}
{}
\newtheorem{theorem}{Theorem}

\newtheorem{lemma}{Lemma}
\newtheorem{corollary}{Corollary}
\newtheorem{question}{Question}
\theoremstyle{remark}

\newtheorem{example}{Example}
\newcommand{\C}{\mathbb{C}}
\newcommand{\disk}{\mathbb{D}}
\newcommand{\ball}{\mathbb{B}}
\newcommand{\D}{\Omega}
\newcommand{\ep}{\varepsilon}
\newcommand{\supp}{\text{supp}}
\newcommand{\Dc}{\overline{\Omega}}
\newcommand{\dbar}{\overline{\partial}}
\newcommand{\sumprime}{\sideset{}{'}\sum}
\newcommand{\zb}{\overline{z}}
\title[Compactness of Toeplitz operators on pseudoconvex domains]{Compactness 
	of Toeplitz operators with continuous symbols on pseudoconvex domains in 
	$\mathbb{C}^n$}
\author{Tomas Miguel Rodriguez}
\author{S\"{o}nmez \c{S}ahuto\u{g}lu}
\address{University of Toledo, Department of
Mathematics \& Statistics, Toledo, OH 43606, USA}
\email{trodrig8@rockets.utoledo.edu, Sonmez.Sahutoglu@utoledo.edu}
\thanks{The material in this paper comes from the first author’s PhD thesis 
\cite{RodriguezThesis} written at University of Toledo under the supervision 
of the second author.}
\date{\today}
\begin{document}

\begin{abstract}
Let $\Omega$ be a bounded pseudoconvex domain in $\mathbb{C}^n$ with Lipschitz 
boundary and $\phi$ be a continuous function on $\overline{\Omega}$. We show that 
the Toeplitz operator $T_{\phi}$ with symbol $\phi$ is compact on the weighted Bergman 
space if and only if $\phi$ vanishes on the boundary of $\Omega$. We also show that 
compactness of the Toeplitz operator $T^{p,q}_{\phi}$ on $\overline{\partial}$-closed 
$(p,q)$-forms for $0\leq p\leq n$ and $q\geq 1$ is equivalent to $\phi=0$ on $\Omega$. 
\end{abstract}
\maketitle

The famous Axler-Zheng Theorem (\cite[Theorem 2.2]{AxlerZheng98}) 
characterizes compactness of a finite sum of finite products of Toeplitz operators on 
the Bergman space of the unit disc in terms of the behavior of the Berezin transform 
of the operator. Engli\v{s} in \cite{Englis99} extended this result to irreducible bounded 
symmetric domains. We note that the Axler-Zheng Theorem is also true on the polydisc 
as the irreducibility condition can be removed (see \cite[p. 232]{ChoeKooLee09}). 
Later on, the Axler-Zheng theorem  was generalized to the Toeplitz algebra on the 
unit ball by Su\'arez in \cite{Suarez07}. For more recent 
developments about Axler-Zheng type results, we direct the reader to 
\cite{CuckovicSahutoglu13,CucukovicSahutogluZeytunuc18,MitkovskiSuarezWick13} 
and the references therein. 

In this paper we are interested in understanding compactness of Toeplitz 
operators in terms of the behavior of their symbols on the boundary. We will restrict 
ourselves to symbols that are continuous up to the boundary. Namely, we are interested 
in the following question.

\begin{question}\label{Question1}
Let $\D$ be a bounded domain in $\C^n$ and $\phi\in C(\Dc)$. Assume that 
the Toeplitz operator $T_{\phi}$ is compact on the Bergman space $A^2(\D)$. 
Is $\phi=0$ on the boundary of $\D$?
\end{question} 

It is well known that, due to Montel's Theorem, the converse of Question \ref{Question1} 
is true (see Corollary \ref{CorMultToeplitz} below). Furthermore, in dimension one  
(that is $\D\subset \C$) the answer is yes when $\D$ is a $C^1$-smooth 
bounded domain (see, for instance, \cite[Theorem 2.3]{ArazyEnglis01}). 

In higher dimensions (that is when $\D\subset \C^n$ for $n\geq 2$), Arazy and Engli\v{s}  
\cite[Theorem 2.3]{ArazyEnglis01} showed that the answer to Question \ref{Question1} is yes 
when $\D$ is a $C^3$-smooth bounded strongly pseudoconvex domain. Le 
\cite[Theorem 2.5]{Le10Toeplitz} answered the question in the affirmative when $\D$ is 
the polydisc. \v{C}u\v{c}kovi\'c and \c{S}ahuto\u{g}lu in 
\cite{CuckovicSahutoglu13,CuckovicSahutoglu14Erratum} also answered the question in the 
affirmative when the domain is $C^{\infty}$-smooth bounded and pseudoconvex satisfying 
condition $R$ and the set of strongly pseudoconvex points is dense in the boundary. 
We note that one can remove condition $R$ by adopting the proof of 
\cite[Lemma 15]{CucukovicSahutogluZeytunuc18} using peak points. 
Recently, \v{C}u\v{c}kovi\'c, \c{S}ahuto\u{g}lu, and Zeytuncu 
\cite[Theorem 4]{CucukovicSahutogluZeytunuc18} extended the answer to weighted 
Bergman spaces on $C^2$-smooth bounded strictly pseudoconvex domains. 

\section{Main Results}
Before we state our first result, we introduce some basic notation. 
Let $\D$ be a domain in $\C^n$ and $h:\D\to (0,\infty)$ be a continuous function. 
Then 
\[A^2(\D,h)=\left\{f\in \mathcal{O}(\D): \|f\|_h^2=\int_{\D}|f|^2hdV<\infty\right\}\] 
denotes the weighted Bergman space on $\D$ with weight $h$. Here $\mathcal{O}(\D)$ 
and $dV$ denote the set of holomorphic functions and the Lebesgue measure 
on $\D$, respectively. 

Since $h$ is continuous, $A^2(\D,h)$ is a closed subspace of $L^2(\D,h)$ 
(see, for example, \cite[Theorem 3.1]{Pasternak-Winiarski90}). Then there exists 
an orthogonal projection 
\[P:L^2(\D,h)\to A^2(\D,h)\] 
called the Bergman projection. The Toeplitz operator is 
\[T_{\phi}=PM_{\phi}:A^2(\D,h)\to A^2(\D,h)\]  
with symbol $\phi\in L^{\infty}(\D)$, where 
\[M_{\phi}:L^2(\D,h)\to L^2(\D,h)\] 
is the multiplication operator by $\phi$. We note that 
\[H_{\phi}=(I-P)M_{\phi}:A^2(\D,h)\to L^2(\D,h)\] 
is called the Hankel operator with symbol $\phi$. In this paper, 
we will focus on the weighted spaces with weight $h=|\rho|^r$ 
where $\rho$ is a defining function and $r\geq 0$ is a real number. 

Our first result answers Question \ref{Question1} on a very large class of domains. 
We note that a defining function for a Lipschitz domain is assumed to be Lipschitz.

\begin{theorem}\label{ThmFunc}
Let $\D$ be a bounded pseudoconvex domain in $\C^n$ with Lipschitz boundary 
and a  defining function $\rho$. Assume that the Toeplitz operator $T_{\phi}$ is 
compact on $A^{2}(\D,|\rho|^r)$ for $ \phi \in C(\Dc)$ and a real number $r\geq 0$.   
Then $\phi=0$ on $b\D$. 
\end{theorem}

Theorem \ref{ThmFunc} does not hold in general. In Example \ref{Ex1} 
below we see that boundary regularity cannot be dropped as 
$\mathbb{D}\setminus \{0\}$ is a bounded non-Lipschitz (pseudoconvex) domain. 
In Example \ref{Ex2}, we see that the boundedness of the domain is necessary. 
We note that the domain in Example \ref{Ex2} is non-pseudoconvex. 
So it would be interesting if one could construct a similar example on  
a pseudoconvex domain with a non-trivial yet finite dimensional Bergman space. 
Finally, Example \ref{Ex3} shows that pseudoconvexity is necessary in 
Theorem \ref{ThmFunc}.   

\begin{example} \label{Ex1}
Consider the unit disc, $\mathbb{D}$, and $\mathbb{D}\setminus\{0\}$ in $\C$. 
We note that any domain in the complex plane (and hence  $\mathbb{D}\setminus\{0\}$)  
is pseudoconvex. Let $\phi \in C_{0}^{\infty}(\mathbb{D})$ such that $\phi(0) \neq 0$. 
Then one can show that $T_{\phi}$ is compact on $A^{2}(\mathbb{D}\setminus \{0\})$ 
because any $f\in A^2(\mathbb{D}\setminus \{0\})$ has a holomorphic extension 
to $\mathbb{D}$. 
\end{example}

\begin{example} \label{Ex2}
Let $\D\subset \C^2$ be an unbounded non-pseudoconvex Reinhardt domain such 
that the Bergman space $A^2(\D)$ is non-trivial finite dimensional (see \cite{Wiegerinck84}). 
Since on a finite dimensional space any linear map is compact, $T_{\phi}$ is compact 
on $A^2(\D)$ for any symbol $\phi \in L^{\infty}(\D)$.
\end{example}

\begin{example}\label{Ex3}
Let $\ball_{n,r}$ denote the ball in $\C^n$ centered at the origin with radius $r$ and 
$\D=\ball_{n,2}\setminus\overline{\ball_{n,1}}$ for $n\geq 2$. Then $\D$ is a 
$C^{\infty}$-smooth bounded non-pseudoconvex domain. By the Hartogs theorem 
any function holomorphic on $\D$ has a holomorphic extension onto $\ball_{n,2}$. 
Therefore, any $\phi$ that is bounded and compactly supported on $\ball_{n,2}$ 
(even if it is not zero on the boundary of $\ball_{n,1}$) will produce a compact Toeplitz 
operator $T_{\phi}$. 
\end{example}

Next, we turn our attention to $\dbar$-closed $(p,q)$-forms and prove a version of 
Theorem \ref{ThmFunc}  for $0\leq p\leq n$ and $1\leq q\leq n$ (see Corollary \ref{CorP0} 
below about the case for $q=0$). But first we set the notation.  

Let $K_{(p,q)}^2(\D)$ denote the set of square integrable $\dbar$-closed $(p,q)$-forms,  
which is a closed subspace of $L_{(p,q)}^2(\D)$, the square integrable  $(p,q)$-forms on 
$\D$ for $0\leq p,q\leq n$. We note that a $\dbar$-closed  $(p,0)$-form is composed of 
holomorphic components. So in that case we will use the notation $A^2_{(p,0)}(\D)$. 

Let 
\[P^{p,q}: L_{(p,q)}^2(\D) \rightarrow K_{(p,q)}^2(\D)\] 
denote the Bergman projection on $(p,q)$-forms. 
The Toeplitz operator $T^{p,q}_{\phi}:K_{(p,q)}^2(\D) \rightarrow K_{(p,q)}^2(\D)$  with 
symbol $\phi \in L^{\infty}(\D)$ is defined as
\[T^{p,q}_{\phi}f = P^{p,q}(\phi f)\] 
for $f \in K_{(p,q)}^2(\D)$. In a similar fashion, the Hankel operator 
$H^{p,q}_{\phi}: K_{(p,q)}^2(\D) \rightarrow L_{(p,q)}^2(\D)$ 
with symbol $\phi \in L^{\infty}(\D)$ is defined as 
\[H^{p,q}_{\phi}f = (I - P^{p,q})(\phi f) = \phi f - P^{p,q}(\phi f)\]
for $f \in K_{(p,q)}^2(\D)$. We note that  
\begin{align}\label{multiplication}
M^{p,q}_{\phi} = T^{p,q}_{\phi} + H^{p,q}_{\phi}.
\end{align}
where $M^{p,q}_{\phi}$ is the multiplication operator from  $K_{(p,q)}^2(\D)$ 
to  $L_{(p,q)}^2(\D)$. Furthermore, $H^{p,q}_{\phi}$ 
and $T^{p,q}_{\phi}$ are compact if and only if $M^{p,q}_{\phi}$ is compact. 

Theorem \ref{ThmFunc} shows that for $\phi \in C(\Dc)$ compactness of the Toeplitz 
operator $T_{\phi}$ on $A^2(\D)$ is equivalent to $\phi$ vanishing on the boundary. 
However,  the situation is much more drastic on $\dbar$-closed forms. Namely, 
compactness of the Toeplitz operator $T^{p,q}_{\phi}$  for $q\geq 1$ is equivalent 
to $\phi=0$ on the whole domain.  

\begin{theorem}\label{ThmForms}
Let $\D$ be a bounded pseudoconvex domain in $\C^n, 0\leq p\leq n,$ and 
$1\leq q\leq n$. Assume that $\phi \in C(\D) \cap L^{\infty}(\D)$. Then 
$T^{p,q}_{\phi}$ is compact on $K_{(p,q)}^2(\D)$ if and only if $\phi = 0$ on $\D$.
\end{theorem}

The theorem above has the following immediate corollary.

\begin{corollary}
Let $\D$ be a bounded pseudoconvex domain in $\C^n, 0\leq p\leq n,$ and 
$1\leq q\leq n$. Assume that $\phi \in C(\D) \cap L^{\infty}(\D)$. Then the 
following are equivalent 
\begin{itemize}
\item[i.] the multiplication operator $M^{p,q}_{\phi}$ is compact on $K_{(p,q)}^2(\D)$,
\item[ii.] the Toeplitz operator $T^{p,q}_{\phi}$ is compact on $K_{(p,q)}^2(\D)$, 
\item[iii.]  $\phi = 0$ on $\D$.
\end{itemize}
\end{corollary}

The rest of the paper is organized as follows: In the next section we prove some 
useful lemmas. In Section \ref{SecProofThmMain} we present the proof of 
Theorem \ref{ThmFunc} and state some corollaries. 
We end the paper with the proof of Theorem \ref{ThmForms}.
  
\section{Background} \label{SecBackground} 
Throughout the paper $\D$ is assumed to be a bounded domain and 
$h:\D\to (0,\infty)$ is a continuous function so that the weighted 
Bergman space $A^2(\D, h)$ is non-trivial. Then by Riesz representation 
theorem, there exists the weighted Bergman kernel $K^h:\D\times\D\to \C$ 
such that $K^h_z=K^h(\cdot,z)\in A^2(\D, h)$ and 
$f(z)=\langle f,K^h_z\rangle_{L^2(\D, h)}$ for any $f\in A^2(\D, h)$. 
The Cauchy-Schwarz Inequality implies that  
\begin{align}\label{CauchySchwarz}
|f(z)| \leq \|f\|_{L^2(\D,h)}\sqrt{K^h(z,z)}
\end{align}
for $f \in A^2(\D,h)$ and $z \in \D$. Furthermore, 
\begin{align}\label{maximum}
K^h(z,z)= \text{max}\{|f(z)|^2: f \in A^2(\D,h), \|f\|_{L^2(\D, h)}\leq 1\} 
\end{align}
for any $z \in \D$. 

We note that we will drop the superscript to denote the unweighted Bergman 
kernel. That is, for $h=1$ we denote the (unweighted) Bergman kernel by 
$K(w,z)=K_z(w)$.  The corresponding normalized Bergman kernel $k_z$ is 
defined as 
\[k_z(w)=\dfrac{K(w,z)}{\sqrt{K(z,z)}}\] 
for $w,z \in \D$. That is, $\|k_z\|_{L^2(\D)}=1$. 

We say that $g_k \rightarrow 0$ weakly as 
$k \rightarrow \infty$ when $\langle f,g_k \rangle \rightarrow 0$ as  
$k \rightarrow \infty$ for all $f \in A^2(\D)$. We note that, when 
$\D \subset \C^n$ is a $C^\infty$-smooth bounded pseudoconvex or a bounded 
convex domain, the normalized Bergman kernel $k_{z} \rightarrow 0$ weakly 
as $z \rightarrow b\D$ (see \cite[Lemma 4.9]{CuckovicSahutolgu18} or 
\cite[Lemma 5.6]{Tikaradze15}). We generalize this fact in the following lemma.

\begin{lemma}\label{LemWeakLip}
Let $\D$ be a bounded pseudoconvex domain in $\C^n$ with Lipschitz boundary 
and $p \in b\D$, then $k_z \rightarrow 0$ weakly as $z \rightarrow p$.
\end{lemma}

\begin{proof}
We note that a bounded pseudoconvex domain with Lipschitz boundary in $\C^n$ is 
hyperconvex (see \cite{Demailly87} and \cite[Theorem* 12.4.6]{JarnickiPflugBook2ndEd}). 
For $z\in \D$ we define 
\[\Gamma_z=\left\{w \in \D : g_{\D}(z,w) < \dfrac{1}{e}\right\}\] 
where $g_{\D}(z,w)$ is the pluricomplex Green function. 
By  \cite[Theorem 12.4.3]{JarnickiPflugBook2ndEd} (see also  \cite[Proposition 3.6]{Herbort99}), 
there exists a constant $C_{\D}>0$ such that for $f \in A^2(\D)$ and $z\in \D$ we have
\[|\langle f,k_z \rangle_{L^2(\D)}|^2 
=  \dfrac{|f(z)|^2}{K(z,z)}\leq C_{\D}\int_{\Gamma_z} |f|^2dV.\]
On the other hand, as $z \rightarrow p$, the Lebesgue measure of $\Gamma_z$ goes to $0$  
(see \cite[Theorem 12.4.4]{JarnickiPflugBook2ndEd}). 
 Then $\lim_{z \rightarrow p}\langle f,k_z \rangle_{L^2(\D)}=0$. That is, 
$k_z\rightarrow 0$ weakly as $z \rightarrow p$.  
\end{proof}

We note that without any regularity assumption of the boundary, Lemma \ref{LemWeakLip} 
is not true. See \cite[Remark 1]{CuckovicSahutoglu21} for a counterexample.

\begin{lemma}\label{LemToeplitz}	
Let $\D$ be a bounded domain in $\C^n,\psi \in C(\Dc),\phi\in L^{\infty}(\D)$,  
and $h:\D\to (0,\infty)$ be a continuous function. Assume that $T_{\phi}$ is compact 
on $A^2(\D,h)$. Then $T_{\phi\psi}$ is compact on $A^2(\D,h)$. 
\end{lemma}

\begin{proof}
We recall the connection between Toeplitz and Hankel Operators,
\begin{align}
\label{Lem1Eqn1}T_{f_1}T_{f_2}=T_{f_1f_2}-H^*_{\overline{f}_1}H_{f_2}
\end{align}
where $H_{\overline{f}_1}$ and $H_{f_{2}}$ are Hankel operators with symbols 
$\overline{f}_1$ and $f_{2}$, respectively (see, for instance, equation 
	(8.6) in \cite{ZhuBook}).

We consider the Toeplitz operators  $T_{z^L}$ and $T_{\zb^K}$ on $A^2(\D,h)$ for 
$L,K\in \mathbb{N}_{0}^{n}$ where $\mathbb{N}_0=\{0,1,2,\ldots\}$. We observe 
that $T_{z^L}$ and $T_{\zb^K}$ are bounded operators. Then $T_{\zb^K}T_{\phi}T_{z^L}$ 
is compact as $T_{\phi}$ is compact.   

Now, we want to show that $T_{\zb^K}T_{\phi}T_{z^L}=T_{\zb^K\phi z^L}$.  
We let $g \in A^2(\D,h)$ and consider $T_{\phi}T_{z^L}$. Then
\[T_{\phi}T_{z^L}g=T_{\phi}P(z^Lg)=T_{\phi}(z^Lg)=P(\phi z^Lg)=T_{\phi z^L}g.\]
From \eqref{Lem1Eqn1} and the fact that $H_{z^K}=0$, we have, 
\begin{align*}
T_{\zb^K}T_{\phi}T_{z^L}g = T_{\zb^K}T_{\phi z^L}g 
= T_{\zb^K\phi z^L}g - H_{z^K}^*H_{\phi z^L}g 
= T_{\zb^K\phi z^L}g.
\end{align*}   
Hence, $T_{\zb^K\phi z^L}$ is a compact operator for all $K,L\in \mathbb{N}_{0}^{n}$. 
Since $\psi \in C(\Dc)$, by the Stone-Weierstrass Theorem, there exists a sequence 
of polynomials $\{P_{n}(z,\zb)\}_{n \in \mathbb{N}}$ such that 
$P_{n} \to \psi$ uniformly on $\Dc$ as $n\to \infty$. The compactness of an operator is 
preserved under convergence in operator norm. Since $T_{\phi P_{n}}$ is compact for 
each $n$ and $T_{\phi P_{n}} \rightarrow T_{\phi \psi}$ in operator norm as $n\to\infty$, 
the operator $T_{\phi\psi}$ is compact on $A^2(\D,h)$.
\end{proof}

\begin{lemma}\label{LemKernel}
Let $\D$ be a bounded domain in $\C^n$, $U$ be a subdomain in $\D$,  
and $h:\D\to (0,\infty)$ be a continuous function such that 
$K_{\D}^h(q,q)>0$ for $q\in U$.  Then for  $q \in U$, we have 
\[ \dfrac{K_{\D}^h(q,q)}{K_U^h(q,q)} \leq \|k_q^h\|_{L^2(U,h)}^2\] 
where $k_q^h=\frac{K_{\D}^h(\cdot,q)}{\sqrt{K_{\D}^h(q,q)}}$ 
is the normalized Bergman kernel of $\D$.
\end{lemma} 

\begin{proof}
Let  $q\in U$. Then
\begin{align*}
\sqrt{K_{\D}^h(q,q)}  = \langle k_q^h, K_{U,q}^h \rangle_{L^2(U,h)} 
\leq \|k_q^h\|_{L^2(U,h)}\sqrt{K_U^h(q,q)}.
\end{align*}
Therefore, we have 
\[\dfrac{K_{\D}^h(q,q)}{K_U^h(q,q)} \leq \|k_q^h\|_{L^2(U,h)}^2 \] 
and the proof of the lemma is complete. 
\end{proof}
The next lemma shows that on a bounded domain, the symbol vanishing on the 
boundary is sufficient for compactness of the Toeplitz operator.

\begin{lemma}\label{LemCompact}
Let $\D$ be a bounded domain in $\C^n$ and $h:\D\to (0,\infty)$ be a 
continuous function. Assume that $\phi\in C(\Dc)$ such that $\phi=0$ on $b\D$. 
Then the multiplication operator $M_{\phi}$ is compact on $A^2(\D,h)$. 
\end{lemma}

\begin{proof}
Since $\phi = 0$ on $b\D$, for $\ep > 0$ there exists $K$ precompact in $\D$ 
such that $|\phi|^2 < \ep$ on $\D \setminus K$. Then, for $f \in A^2(\D,h)$ we have
\begin{align*}
\|M_{\phi}f\|_{L^2(\D,h)}^2  
= &\, \int_{\D \setminus K} |\phi f|^2hdV + \int_K |\phi f|^2hdV \\
\leq &\, \ep\|f\|_{L^2(\D,h)}^2 + \|M_{\phi\chi_K} f\|_{L^2(\D,h)}^2. 
\end{align*}
Since (by Montel's Theorem) $M_{\phi\chi_K}$ is compact on $A^2(\D,h)$, the inequality 
above shows that $M_{\phi}$ satisfies the compactness estimates in Lemma \ref{LemCompEst} 
below. Therefore, $M_{\phi}$ is compact on $A^2(\D,h)$.
\end{proof}

We note that if the multiplication operator $M_{\phi}$ is compact on $A^2(\D,h)$ then 
the Toeplitz operator  $T_{\phi}=PM_{\phi}$ is compact as well. Hence the following 
corollary is an immediate consequence of the lemma above.

\begin{corollary}\label{CorMultToeplitz}
Let $\D$ be a bounded domain in $\C^n$ and $h:\D\to (0,\infty)$ be a 
continuous function. Assume that $\phi\in C(\Dc)$ such that $\phi=0$ on $b\D$. 
Then the Toeplitz operator $T_{\phi}$ is compact on $A^2(\D,h)$. 
\end{corollary}
The following corollary is a consequence of the corollary above and Theorem \ref{ThmFunc}. 
\begin{corollary}\label{CorMult}
Let $\D$ be a bounded pseudoconvex domain in $\C^n$ with Lipschitz boundary 
and a defining function $\rho$. Assume that $r\geq 0$ is a real number and 
$\phi \in C(\Dc)$. Then the following are equivalent
\begin{itemize}
\item[i.] the multiplication operator $M_\phi$ is compact on $A^2(\D,|\rho|^r)$,
\item[ii.] the Toeplitz operator $T_{\phi}$ is compact on $A^2(\D,|\rho|^r)$,
\item[iii.] $\phi=0$ on $b\D$. 
\end{itemize}
\end{corollary}

\section{Proof of Theorem \ref{ThmFunc}} \label{SecProofThmMain}
We denote the Bergman kernel for $A^2(\D,|\rho|^r)$ by $K^{r,\rho}$ and 
\[k_z^{r,\rho}(w)=\dfrac{K^{r,\rho}(w,z)}{\sqrt{K^{r,\rho}(z,z)}}\]  
denotes the corresponding normalized Bergman  kernel. When we need to be specific 
about the domain, we will denote the Bergman kernel by $K_{\D}^{r,\rho}$. 

Let $T:A^2(\D,|\rho|^r)\to A^2(\D,|\rho|^r)$ be a bounded linear operator. 
Then the Berezin transform $\widetilde{T}(z)$ of $T$ at $z\in \D$ is defined as 
\[\widetilde{T}(z)=\langle Tk_z^{r,\rho},k_z^{r,\rho} \rangle_{L^2(\D,|\rho|^r)}\] 
where $\langle \cdot, \cdot \rangle_{L^2(\D,|\rho|^r)}$ is the inner product on 
$L^2(\D,|\rho|^r)$. For $\phi \in L^{\infty}(\D)$, we denote 
$\widetilde{\phi}=\widetilde{T_{\phi}}$.

We first prove a result similar to Lemma \ref{LemWeakLip} for the normalized 
Bergman kernel $k_z^{r,\rho}$ of the weighted Bergman space $A^2(\D, |\rho|^r)$ 
as shown in Lemma \ref{LemWeakConv} below. In the proof, we use the notion 
of inflated domain over the base domain $\D$ (see \cite{ForelliRudin1974}  
and \cite{Ligocka89}).  We define the inflated domain as 
\begin{align}\label{EqnInflatedDomain} 
\D_{r}^m=\{(z,w) \in \D \times \C^m:\rho(z) + |w_1|^{\frac{2m}{r}} + \cdots 
	+ |w_m|^{\frac{2m}{r}} < 0\}.
\end{align}
We denote the trivial extension of $f \in A^2(\D,|\rho|^r)$ to $\D_{r}^m$ by $F$. Namely, 
$F\in A^2(\D_{r}^m)$ such that $f(z)=F(z,w)$ for $(z,w)\in \D_{r}^m$.

\begin{lemma}\label{LemWeakConv}
Let $\D$ be a bounded pseudoconvex domain in $\C^n$ with Lipschitz boundary 
and a defining function $\rho$. Assume that $r\geq 0$ is a real number and $p \in b\D$. 
Then $k_{z}^{r,\rho} \rightarrow 0$ weakly as $z \rightarrow p$. 
\end{lemma}

\begin{proof} 
First, we will produce a bounded pseudoconvex domain with Lipschitz boundary inflated 
over $\D$.  The domain in \eqref{EqnInflatedDomain} is bounded but not necessarily 
Lipschitz  or pseudoconvex, in general. We construct the inflated domain $\widehat{\D}$ 
as follows:  Since $\D$ is a bounded Lipschitz pseudoconvex domain, we use 
\cite[Theorem 1.1]{Harrington2008} to choose a strictly plurisubharmonic 
function $\rho_0$ and $0 < \eta_0 \leq 1\leq C$ such that 
$\frac{1}{C}|d_{\D}|^{\eta_0}<-\rho_0<C|d_{\D}|^{\eta_0}$ on $\D$ where $d_{\D}$ is the negative 
boundary distance function. Then \cite[Proposition 3.1]{Harrington2022} 
implies that for $0<\eta<\eta_0$ we can choose a Lipschitz defining function 
$\widehat{\rho}$ for $\D$ such that $\widehat{\rho}\in C^{\infty}(\D)$
 and $-(-\widehat{\rho})^{\eta}$ is strictly plurisubharmonic on $\D$.

Let $s=r/\eta \geq 0$ and $m \in \mathbb{N}$ such that $s \leq m$. 
With this, we define the inflated domain 
\[\widehat{\D}=\{(z,w) \in \D \times \C^m: -(-\widehat{\rho}(z))^{\eta} 
+ |w_1|^{\frac{2m}{s}} + \cdots + |w_m|^{\frac{2m}{s}}  < 0\}.\] 
Then, $\widehat{\D}$ is a bounded pseudoconvex domain in $\C^{n+m}$. 
Next we want to show that $\widehat{\D}$ has a Lipschitz boundary as well. 

We observe that,
\begin{align*}
\widehat{\D} =& \{(z,w) \in \D \times \C^m: -(-\widehat{\rho}(z))^{\eta}   
	+ |w_1|^{\frac{2m}{s}} + \cdots + |w_m|^{\frac{2m}{s}}  < 0\} \\
=& \{(z,w) \in \D \times \C^m: \widehat{\rho}(z) + \lambda(w) < 0\}
\end{align*}	
where 
\[\lambda(w) = \left(|w_1|^{\frac{2m}{s}} + \cdots + |w_m|^{\frac{2m}{s}}\right)^{1/\eta}.\] 

One can check that $\lambda$ has uniformly bounded first order partial derivatives 
on  $\widehat{\D}$. Then $\lambda$ is a Lipschitz function. Hence $\widehat{\D}$ has 
Lipschitz boundary away from the boundary of $\D$. On the other hand, for a boundary 
point $p\in b\D$ (after a holomorphic linear change of coordinates, if necessary) there 
exist an open neighborhood $U$ of $p$ in $\C^n$ and a Lipschitz function $\varphi$ such that 
$\widehat{\rho}(z)=y_n-\varphi(z',x_n)$ on $U$ where $z'=(z_1,\ldots, z_{n-1})$ 
and $z_n=x_n+iy_n$. Then 
\[U\cap \D=\{z\in U: y_n<\varphi(z',x_n)\}\]  
and 
\[(U\times \C^m)\cap \widehat{\D}
	= \left\{(z,w)\in U\times \C^m: y_n<\varphi(z',x_n)-\lambda(w)\right\}.\]
Hence, $\widehat{\D}$ has Lipschitz boundary near the boundary of $\D$. 
Therefore, $\widehat{\D}$ is a bounded pseudoconvex 
domain in $\C^{n+m}$ with Lipschitz boundary and, by Lemma \ref{LemWeakLip}, we have 
$k_{(z,w)}^{\widehat{\D}} \rightarrow 0$ weakly as $(z,w) \rightarrow b\widehat{\D}$. 

Let 
\[\D_z=\left\{w\in \C^m:  |w_1|^{\frac{2m}{s}} + \cdots + |w_m|^{\frac{2m}{s}} 
	<|\widehat{\rho}(z)|^{\eta}\right\}\] 
be the fiber over $z\in \D$. For any $z\in \D$,  using the change of coordinates 
\[\widetilde{w}_j=|\widehat{\rho}(z)|^{-r/{2m}}w_j,\]   
one can show that the volume of $\D_z$ is 
\[V(\D_z)=c_{m,s}|\widehat{\rho}(z)|^r\]
where 
\[c_{m,s}=\int_{|\widetilde{w}_1|^{\frac{2m}{s}}+ \cdots + |\widetilde{w}_m|^{\frac{2m}{s}} 
< 1}dV(\widetilde{w}).\] 
Since $\widehat{\D}$ is a Hartogs domain over the base $\D$, following a similar 
argument as in the proof of \cite[Proposition 8]{CucukovicSahutogluZeytunuc18}, 
we derive
\begin{align*}
 K_{\D}^{r,\widehat{\rho}}(\xi,z)=c_{m,s}K_{\widehat{\D}}(\xi,0;z,0). 
\end{align*}  
Then for $f\in A^2(\D,|\widehat{\rho}|^r)$ and $F \in A^2(\widehat{\D})$ related 
to $f$ as described after \eqref{EqnInflatedDomain} we have
\begin{align*} 
 \langle F, k_{(z,0)}^{\widehat{\D}} \rangle_{L^2(\widehat{\D})} 
=& \frac{1}{\sqrt{K_{\widehat{\D}}(z,0;z,0)}} \int_{\xi\in\D} f(\xi)  
	\int_{\tau\in\D_{\xi}} \overline{K_{\widehat{\D}}(\xi,\tau;z,0)}dV(\tau)dV(\xi)\\
=& \frac{1}{\sqrt{K_{\widehat{\D}}(z,0;z,0)}} \int_{\xi\in\D} f(\xi) 
\overline{K_{\widehat{\D}}(\xi,0;z,0)} V(\D_{\xi}) dV(\xi)\\
=& \frac{c_{m,s}}{\sqrt{K_{\widehat{\D}}(z,0;z,0)}} \int_{\xi\in\D} f(\xi) 
\overline{K_{\widehat{\D}}(\xi,0;z,0)} |\widehat{\rho}(\xi)|^r dV(\xi)\\
=&\sqrt{c_{m,s}} \langle f, k_{z}^{r,\widehat{\rho}} \rangle_{L^2(\D, |\widehat{\rho}|^r)}. 
\end{align*} 
Since $k_{(z,w)}^{\widehat{\D}} \rightarrow 0$ 
weakly as $(z,w) \rightarrow b\widehat{\D}$, the equality above shows that 
$k_{z}^{r,\widehat{\rho}} \rightarrow 0$ weakly as $z \rightarrow p$. Lastly, since 
$\widehat{\rho}$ and $\rho$ are comparable on $\Dc$, by 
\cite[Lemma 17]{CucukovicSahutogluZeytunuc18} we conclude that 
$k_{z}^{r,\rho} \rightarrow 0$ weakly as $z \rightarrow p$.
\end{proof}

Aside from the weak convergence of the normalized Bergman kernel, a crucial ingredient in 
proving Theorem \ref{ThmFunc} is the use of localization of the Bergman kernel  
(see \cite{Ohsawa84}). We prove a version of this localization 
result for $A^2(\D,|d_{\D}|^r)$ where $d_{\D}$ is the negative boundary distance function. 

\begin{theorem}\label{ThmLocal}
Let $\D$ be a bounded pseudoconvex domain in $\C^n$. Assume that 
$r\geq 0$ is a real number and $U_1 \Subset U_2$ are open sets in $\C^n$ such that 
$\D\cap U_2\neq\emptyset$. Then there exists $C>0$ such that for any connected 
component $V$ of $\D \cap U_2$ we have  
\[K_{V}^{r,d_{\D}}(z,z) \leq C K_{\D}^{r,d_{\D}}(z,z)\]
for $z \in V \cap U_1$. 
\end{theorem}

\begin{proof}
Let $\chi: \C^n \rightarrow [0,1]$ be a $C^{\infty}$-smooth function such that $\chi=1$ in 
some neighborhood $\widetilde{U}_1$ of $\overline{U}_1$ and $\supp(\chi) \subset U_2$. 
We fix a connected component $V \subset \D \cap U_2$ and a point $z_{0} \in V \cap U_1$. 

By equation \eqref{maximum}, there exists an $f \in A^2(V,|d_{\D}|^r)$ such that  
$|f(z_0)|^2=K_V^{r,d_{\D}}(z_0,z_0)$ and $\|f\|_{L^2(V, |d_{\D}|^r)}=1$. In fact, 
$f=k_{V,z_0}^{r,d_{\D}}$. We extend $f=0$ trivially on $\D \setminus V$ and	define 
\[\alpha=
\begin{cases}
f\dbar\chi & \text{on } U_2 \cap \D \\
0 &  \text{on }  \D \setminus U_2   
\end{cases}. \]
We observe that $\alpha$ is a $C^{\infty}$-smooth $\dbar$-closed $(0,1)$-form 
on $\D$.  By Oka's Lemma  (see \cite[Theorem 3.4.10]{ChenShawBook}), the 
function $-\log|d_{\D}(z)|$ is plurisubharmonic on $\D$. Hence 
\[\psi(z)= 2n\log|z-z_0|-r\log|d_{\D}(z)|\] 
 is plurisubharmonic as well. 

Now, there exist $C_1,C_2 > 0$, independent of $f$ and $V$, such that
\begin{align*}
\int_{\D}|\alpha(z)|^2e^{-\psi(z)}dV(z) 
\leq C_1\int_{(\D\cap U_2)\setminus \widetilde{U}_1}|z-z_0|^{-2n}|f(z)|^2|d_{\D}(z)|^rdV(z) 
\leq C_2.
\end{align*}   
By \cite[Theorem 2.2.3]{Hormander65} and the previous inequality, we can find 
$g \in C^{\infty}(\D)$ such that $\dbar g = \alpha$ and 
\begin{align}\label{estimatez_0}
\int_{\D}|g(z)|^2e^{-\psi(z)}dV =\int_{\D}|g(z)|^2|z-z_0|^{-2n}|d_{\D}(z)|^rdV(z)\leq C_3
\end{align}
where $C_3$ depends on $C_2$ and the diameter of $\D$. We define 
$\widehat{f}=\chi f -g$. Then $\widehat{f}$ is holomorphic on $\D$ because 
$\dbar \widehat{f}=\dbar(\chi f) - \dbar g=0$. Also, by equation \eqref{estimatez_0}, 
we have $g(z_0)=0$ (otherwise the integral is not finite) which implies 
that $\widehat{f}(z_0)=f(z_0)$. Furthermore,
\begin{align*}
\|\widehat{f}\|_{L^2(\D,|d_{\D}|^r)} 
\leq \|f\|_{L^2(V,|d_{\D}|^r)} + \|g\|_{L^2(\D,|d_{\D}|^r)} 
\leq 1 + \text{diam}(\D)^n\sqrt{C_3}.
\end{align*} 
By equation \eqref{CauchySchwarz}, we have 
\[K_{V}^{r,d_{\D}}(z_0,z_0) = |f(z_0)|^2= |\widehat{f}(z_0)|^2 
\leq \|\widehat{f}\|_{L^2(\D,|d_{\D}|^r)}^2K_{\D}^{r,d_{\D}}(z_0,z_0) 
\leq C K_{\D}^{r,d_{\D}}(z_0,z_0)\] 
where $C=(1 + \text{diam}(\D)^n\sqrt{C_3})^2$. 	Therefore, the proof of 
Theorem \ref{ThmLocal} is complete. 
\end{proof}

If $\rho$ is another Lipschitz defining function (the defining function $d_{\D}$ 
is also Lipschitz), there exist a positive function $\alpha$ and constants $C_1,C_2 >0$ 
such that $\rho=\alpha d_{\D}$ and $C_1 \leq \alpha \leq C_2$ on $\Dc$ 
(see \cite[Lemmas 0.1 and 0.2]{Shaw05}). Thus, the measures $|\rho|^rdV$ and $|d_{\D}|^rdV$ 
are comparable. That is,  $A^2(\D,|\rho|^r)$ and $ A^2(\D,|d_{\D}|^r)$ are the same as sets 
and there exists $C>0$ such that 
\[\dfrac{\|f\|_{L^2(\D,|d_{\D}|^r)}}{C} 
\leq \|f\|_{L^2(\D,|\rho|^r)} \leq C\|f\|_{L^2(\D,|d_{\D}|^r)}\]
for any $f\in A^2(\D,|\rho|^r)$. 

By \eqref{maximum}, we conclude that the Bergman kernels on the diagonal $K^{r,\rho}(z,z)$ 
and $K^{r,d_{\D}}(z,z)$ are equivalent. Namely, there exists $D > 0$ such that 
\[\dfrac{K^{r,d_{\D}}(z,z)}{D} \leq K^{r,\rho}(z,z) \leq DK^{r,d_{\D}}(z,z).\]
With this, we have the following corollary to Theorem \ref{ThmLocal}. 

\begin{corollary}\label{CorollaryLocal}
Let $\D$ be a bounded pseudoconvex domain in $\C^n$ with Lipschitz boundary 
and a defining function $\rho$. Assume that $r\geq 0$ is a real number 
and  $U_1 \Subset U_2$ are open sets in $\C^n$ such that $\D\cap U_2\neq\emptyset$.  
Then there exists $C>0$ such that for any connected component $V$ 
of $\D \cap U_2$ we have 
\[K_{V}^{r,\rho}(z,z) \leq C K_{\D}^{r,\rho}(z,z)\]
for $z \in V \cap U_1$. 
\end{corollary}

Now, we prove Theorem \ref{ThmFunc}.

\begin{proof}[Proof of Theorem \ref{ThmFunc}] 
Let $\D$ be a bounded pseudoconvex domain with Lipschitz boundary 
and a defining function $\rho$. If $T_{\phi}$ is compact on $A^{2}(\D,|\rho|^r)$, 
then by	Lemma \ref{LemToeplitz} with $\psi=\overline{\phi}$, the Toeplitz 
operator $T_{|\phi|^2}$ is also compact on 	$A^{2}(\D,|\rho|^r)$. 

Let $p \in b\D$. We claim that $\phi(p) = 0$. Suppose not, then there exist $r',\ep > 0$ 
such that $U = B(p,r') \cap \D$ is a domain and $|\phi(z)|^2 > \ep$ for all $z \in U$. 
Consider a sequence $\{q_{j}\} \subset U$ such that $q_{j} \rightarrow p$. 
Lemma \ref{LemWeakConv} and the fact that $T_{|\phi|^2}$ is compact imply that  
 $T_{|\phi|^2}k_{q_{j}}^{r,\rho}\to 0$ in $A^2(\D,|\rho|^r)$ as $j\to \infty$. Then 
\begin{align}\label{EquationBerezinwithweights} 
\lim_{j \rightarrow \infty} \langle T_{|\phi|^2}k_{q_{j}}^{r,\rho}, k_{q_{j}}^{r,\rho} 
\rangle_{L^2(\D,|\rho|^r)} 
= 0.  
\end{align}
On the other hand, by Corollary \ref{CorollaryLocal} and Lemma \ref{LemKernel} 
there exists $C_{U}>0$ such that
\[ \|k_{q_{j}}^{r,\rho}\|_{L^2(U,|\rho|^r)}^2 
\geq \dfrac{K_{\D}^{r,\rho}(q_{j},q_{j})}{K_{U}^{r,\rho}(q_{j},q_{j})} 
\geq C_U > 0 \text{ as } j\to\infty .\]
Thus, 			
\begin{align*}
\langle T_{|\phi|^2}k_{q_{j}}^{r,\rho}, k_{q_{j}}^{r,\rho} 
	\rangle_{L^2(\D,|\rho|^r)} 
\geq  \langle |\phi|^2k_{q_{j}}^{r,\rho}, k_{q_{j}}^{r,\rho} 
	\rangle_{L^2(U,|\rho|^r)} 
> \ep \|k^{r,\rho}_{q_{j}}\|_{L^2(U,|\rho|^r)}^2
\geq \ep C_U.
\end{align*} 
This contradicts \eqref{EquationBerezinwithweights}. Thus, $\phi =0$ on $b\D$. 
\end{proof}	

Next, we prove a localization result about compactness of Toeplitz operators 
with continuous symbols. In the corollary below, to emphasize the domain, we will 
use $T_{\phi}^U$ to denote the Toeplitz operator on $A^2(U)$ with symbol $\phi$. 
We note that the space $A^2(\D,|\rho|^r)$ is the unweighted Bergman space 
when $r=0$. 

\begin{corollary}\label{CorLocal}
Let $\D$ be a bounded pseudoconvex domain in $\C^n$ with Lipschitz boundary 
and $\phi \in C(\Dc)$. 
\begin{itemize}
\item[i.] Assume that $T_{\phi}^{\D}$ is compact on $A^2(\D),p\in b\D,$ and $r>0$ small 
enough so that $U =B(p,r) \cap \D$ is a domain. Then the operator $T_{\phi \chi}^{U}$ is 
compact on $A^2(U)$ for all $\chi \in C_{0}(B(p,r))$.
\item[ii.] Assume that for any $p \in b\D$ there exists $r > 0$ such that $U=\D\cap B(p,r)$ is a 
domain and $T_{\phi \chi}^{U}$ is compact on $A^2(U)$ for all $\chi\in C^{\infty}_0(B(p,r))$. 
Then  $T_{\phi}^{\D}$ is compact on $A^2(\D)$. 
\end{itemize} 
\end{corollary}

\begin{proof} To prove i. we first denote $b_1U =  bB(p,r) \cap \Dc$ and 
$b_2U=\overline{B(p,r)} \cap b\D$. Then $bU = b_1U \cup b_2U$. We observe 
that $\phi \chi = 0$ on $b_1U$ for all $\chi \in C_{0}(B(p,r))$ as $\chi = 0$ on $b_1U$. 
Since $T_{\phi}^{\D}$ is compact on $A^2(\D)$, Corollary \ref{CorMult} implies that $\phi = 0$ 
on $b\D$. Thus, $\phi \chi =0$ on $b_2U$. Hence, $\phi\chi = 0$ on $bU = b_1U \cup b_2U$. 
Therefore, by Corollary \ref{CorMultToeplitz}, $T_{\phi \chi}^{U}$ is compact on $A^2(U)$.

To prove ii. we cover $b\D$ by finitely many balls $\{B(p_j,r_j):j=1,\ldots,m\}$ such that 
each $p_j\in b\D$ and $r_j>0$ guaranteed by the hypothesis of ii.  
We choose a $C^{\infty}$-smooth partition of unity 
$\{\chi_j : j=1,\ldots, m\}$ subordinate to the balls $\{B_j : j=1,\ldots, m\}$. Then 
\begin{align}\label{partition}
\phi=\phi\sum_{j=1}^m\chi_j=\sum_{j=1}^m\phi\chi_j\ \text{on}\ b\D.
\end{align}
Let us denote $U_j=B_j \cap \D$. Since $T_{\phi \chi_j}^{U_j}$ is compact on $A^2(U_j)$, 
Corollary \ref{CorMult} implies that $\phi\chi_j =0$ on $bU_j$ for  $j=1,\ldots,m$.  
Thus, by \eqref{partition} we have $\phi (z)=0$ for all $z \in b\D \subset \cup_{j=1}^m bU_j$. 
Finally, Corollary \ref{CorMultToeplitz} implies that $T_{\phi}$ is compact on $A^2(\D)$. 
\end{proof}

\begin{lemma}\label{LemLipschitz}
Let $\D\subset \mathbb{R}^n$ be a finite product of bounded domains with Lipschitz 
boundary. Then $\D$ is a bounded domain with Lipschitz boundary. 
\end{lemma} 

\begin{proof}
It is enough to prove the case $\D=\D_1\times \D_2$ as the general case 
follows by induction. 

Let $\D_j\subset \mathbb{R}^{n_j}$ be a bounded domain 
with Lipschitz boundary for $j=1,2$ and $\D=\D_1\times\D_2\subset \mathbb{R}^n$. 
We note that a bounded domain has Lipschitz boundary if and only if it  satisfies the 
uniform cone condition (see \cite[Theorem 1.2.2.2]{GrisvardBook}). Since $\D_j$ is 
Lipschitz, $b\D_j$ has a finite cover $\{U_{j,k}\}$ with corresponding cones $\{C_{j,k}\}$ 
such that $z_j+C_{j,k} \subset \D_j$ for $z_j \in \D_j \cap U_{j,k}$. The collection of 
products $\{U_{1,k_1} \times U_{2,k_2}\}$ is a finite cover for $b\D_1\times b\D_2$. 
Furthermore, one can show that each $C_{1,k_1} \times C_{2,k_2}$ 
contains a cone $C_{k_1,k_2}$ in $\mathbb{R}^n$ as follows. 
Let $\ep_j>0$ and $\nu_j\in \mathbb{R}^{n_j}\setminus\{0\}$ such that 
\[C_{j,k_j} \supseteq \{c_ju_j\in \mathbb{R}^{n_j}:\|u_j-\nu_j\|<\ep_j, 0\leq c_j< \ep_j\}\]  
for $j=1,2$. Then we choose $\ep= \min\{\ep_1,\ep_2\}>0$ and 
\[C_{k_1,k_2}=\left\{c(u_1,u_2)\in\mathbb{R}^n:\sum_{j=1}^2\|u_j-\nu_j\|^2
<\ep^2, 0\leq c< \ep \right\}.\] 
Hence, $z+C_{k_1,k_2} \subset \D$ for 
$z\in \D\cap (U_{1,k_1} \times U_{2,k_2})$. That is, $\D$ satisfies the 
uniform cone condition near $b\D_1\times b\D_2$. We note that it is 
easy to see that the boundary of $\D$ away from $b\D_1\times b\D_2$ is Lipschitz. 
Therefore, $\D$ has a Lipschitz boundary.
\end{proof}

Corollary \ref{CorMult} and Lemma \ref{LemLipschitz} imply the following corollary.

\begin{corollary}\label{CorFiniteProd}
Let $\D=\D_1\times \cdots \times \D_m\subset \C^n$ be a finite product of bounded 
pseudoconvex domains with Lipschitz boundary and $\phi\in C(\Dc)$. Then 
$T_{\phi}$ is compact on $A^2(\D)$ if and only if $\phi=0$ on $b\D$.
\end{corollary}

We note that, with minor modifications, the proof of Theorem \ref{ThmFunc} still works 
on $A^2_{(p,0)}(\D,|\rho|^r)$  for any $0\leq p\leq n$.  Hence,  we have the following 
corollary. 

\begin{corollary}\label{CorP0}
Let $\D$ be a bounded pseudoconvex domain in $\C^n$ with Lipschitz boundary 
and a defining function $\rho$. Assume that $r\geq 0$ is a real number 
and $\phi \in C(\Dc)$. Then the Toeplitz operator $T^{p,0}_{\phi}$ is compact 
on $A^2_{(p,0)}(\D,|\rho|^r)$ if and only if  $\phi=0$ on $b\D$. 
\end{corollary} 

\section{Proof of Theorem \ref{ThmForms}}\label{SecForms}
One of the well-known tools to show compactness of $H^{p,q}_{\phi}$ is  using 
compactness estimates as in the following lemma (see \cite[Lemma 4.3(ii)]{StraubeBook}). 

\begin{lemma}\label{LemCompEst} 
Let $T: X \rightarrow Y$ be a bounded linear map where $X,Y$ are Hilbert spaces. 
The operator $T$ is compact if and only if for every $\ep >0$ there are a Hilbert 
space  $Z_{\ep}$, a compact linear map $S_{\ep}: X \rightarrow Z_{\ep}$, 
and a constant $C_{\ep}>0$ such that 
\begin{align*}
\|Tf\|_{Y} \leq \ep\|f\|_{X} + C_{\ep}\|S_{\ep}f\|_{Z_{\ep}}
\end{align*}
for all $f\in X$. 
\end{lemma}

Let $\phi \in C^1(\Dc)$ and  $0\leq p,q\leq n$ be integers. Then Kohn's formula, 
$P^{p,q}=I-\dbar^{\ast}N_{p,q+1}\dbar$, implies that 
\begin{align}\label{Kohnsformulaimplies}
H_{\phi}^{p,q}f=\dbar^{\ast}N_{p,q+1}(\dbar \phi \wedge f),
\end{align}
where $N_{p,q+1}$ is the $\dbar$-Neumann operator on $L^2_{(p,q+1)}(\D)$, the square 
integrable $(p,q+1)$-forms (see \cite{ChenShawBook}).   

We introduce some lemmas before we prove Theorem \ref{ThmForms}. First, we 
generalize Lemma \ref{LemToeplitz} to $(p,q)$-forms. 

\begin{lemma}\label{LemToeplitzForms}	
Let $\D$ be a bounded domain in $\C^n$, $\phi \in L^{\infty}(\D), 0\leq p,q\leq n$, 
and $\psi \in C(\Dc)$. Assume that $T^{p,q}_{\phi}$ is compact 
on $K^2_{(p,q)}(\D)$. Then $T^{p,q}_{\phi\psi}$ is compact on $K^2_{(p,q)}(\D)$. 
\end{lemma}

\begin{proof}
We recall the connection between Toeplitz and Hankel operators,
\begin{align} \label{Eq1Forms}
T^{p,q}_{f_1}T^{p,q}_{f_2}=T^{p,q}_{f_1f_2}-H^{p,q*}_{\overline{f}_1}H^{p,q}_{f_2}
\end{align}
where $H^{p,q}_{\overline{f}_1}$ and $H^{p,q}_{f_{2}}$ are Hankel operators with symbols 
$\overline{f}_1$ and $f_{2}$, respectively. Then we use computations similar to the ones 
in the proof of Lemma \ref{LemToeplitz} to get 
\[T^{p,q}_{\phi}T^{p,q}_{z^L}g=T^{p,q}_{\phi}P^{p,q}(z^Lg)
=T^{p,q}_{\phi}(z^Lg)=P^{p,q}(\phi z^Lg)=T^{p,q}_{\phi z^L}g\] 
for $g\in K^2_{(p,q)}(\D)$. 
From \eqref{Eq1Forms}, we have, 
\begin{align*}
T^{p,q}_{\zb^K}T^{p,q}_{\phi}T^{p,q}_{z^L}g 
= T^{p,q}_{\zb^K\phi z^L}g - H^{p,q*}_{z^K}H^{p,q}_{\phi z^L} g 
= T^{p,q}_{\zb^K\phi z^L}g.
\end{align*}   
Then $T^{p,q}_{\zb^K\phi z^L}$ is a compact operator for all $K,L\in \mathbb{N}_{0}^{n}$. 
Finally, just like in the proof of Lemma \ref{LemToeplitz}, we use an argument 
involving the Stone-Weierstrass Theorem to complete the proof.  
\end{proof}

We note that, if $\phi \in C(\Dc)$ such that $\phi=0$ on $b\D$ where $\D$ is a 
bounded convex domain in $\C^n$, the Hankel operator $H^{p,q}_{\phi}$ is compact 
on $K_{(p,q)}^2(\D)$ (see \cite[Lemma 2]{CelikSahutogluStraube20}). The next lemma 
shows that this holds for bounded pseudoconvex domains as well. 

\begin{lemma}\label{compactnesshankel}
Let $\D$ be a bounded pseudoconvex domain in $\C^n$ and $0\leq p,q\leq n$. 
Assume that $\phi \in C(\Dc)$ and that $\phi = 0$ on $b\D$. Then $H^{p,q}_{\phi}$ 
is compact on $K_{(p,q)}^2(\D)$.
\end{lemma}

\begin{proof}
We follow the proof of \cite[Lemma 2]{CelikSahutogluStraube20}. We note that 
$H^{p,n}_{\phi}=0$ as $(p,n)$-forms are $\dbar$-closed. So for the rest of the proof, 
we will assume that $0\leq q\leq n-1$. 

First we assume that $\chi \in C_{0}^\infty(\D)$. Let 
\[f = \sumprime_{|J|=p,|K|=q}f_{JK}dz_J\wedge d\zb_K \in K_{(p,q)}^2(\D)\]
where  $\sumprime$ denotes the sum with strictly increasing indices.  
 By \cite[Theorem 4.4.1]{ChenShawBook} we observe that
\begin{align*}
\dbar\dbar^{\ast} N_{p,q+1}(\dbar \chi \wedge f)=\dbar \chi \wedge f.
\end{align*}
By the preceding equation and \eqref{Kohnsformulaimplies}, we have 
\begin{align} \nonumber
\|H_{\chi}^{p,q}f\|_{L_{(p,q)}^2(\D)}^2 =& \langle  \dbar^{\ast}N_{p,q+1}(\dbar \chi \wedge f), 
\dbar^{\ast}N_{p,q+1}(\dbar \chi \wedge f) \rangle_{L_{(p,q)}^2(\D)} \\
 \label{combine1}=& \langle  N_{p,q+1}(\dbar \chi \wedge f), 
\dbar\dbar^{\ast} N_{p,q+1}(\dbar \chi \wedge f) \rangle_{L_{(p,q+1)}^2(\D)} \\
\nonumber=& \langle  N_{p,q+1}(\dbar \chi \wedge f), \dbar \chi \wedge f \rangle_{L_{(p,q+1)}^2(\D)}.
\end{align}
Now, 
\begin{align}  \label{combine2}
\left|\langle N_{p,q+1}(\dbar \chi \wedge f), \dbar \chi \wedge f \rangle_{L_{(p,q+1)}^2(\D)} \right| 
 \leq \left\| |\nabla \chi| N_{p,q+1}(\dbar \chi \wedge f)\right\|_{L_{(p,q+1)}^2(\D)} 
\|f\|_{L_{(p,q)}^2(\D)}.
\end{align}
Next, we use the inequality $|ab| \leq \dfrac{a^2}{\ep} + \ep b^2$ for real numbers 
$a,b$ and $\ep > 0$  
\begin{align}
\left\| |\nabla \chi| N_{p,q+1}(\dbar \chi \wedge f)\right\|_{L_{(p,q+1)}^2(\D)} 
\|f\|_{L_{(p,q)}^2(\D)} \leq &\dfrac{1}{\ep}\left\| |\nabla \chi| 
N_{p,q+1}(\dbar \chi \wedge f)\right\|_{L_{(p,q+1)}^2(\D)}^2 \\
\nonumber &+ \ep \|f\|_{L_{(p,q)}^2(\D)}^2.
\end{align} 
Now, $|\nabla \chi|$ is a compactly supported continuous function on $\Dc$.  
Thus, it is a compactness multiplier (see \cite{CelikStraube09}). Then, there is a 
constant $C_{\ep}>0$ such that
\begin{align*}
\left\||\nabla \chi| N_{p,q+1}(\dbar \chi \wedge f)\right\|_{L_{(p,q+1)}^2(\D)}^2 
\leq& \ep^2\left\|\dbar N_{p,q+1}(\dbar \chi \wedge f) \right\|_{L_{(p,q+2)}^2(\D)}^2 \\
	&+ \ep^2\left\|\dbar^{\ast} N_{p,q+1}(\dbar \chi \wedge f) \right\|_{L_{(p,q)}^2(\D)}^2\\
& + C_{\ep}\left\|N_{p,q+1}(\dbar \chi \wedge f)\right\|_{-1}^2 
\end{align*}
where $\| \cdot \|_{-1}$ denotes the Sobolev $(-1)$-norm. We observe that 
\begin{itemize}
\item[i.] $\dbar N_{p,q+1}(\dbar \chi \wedge f) = N_{p,q+2}\dbar(\dbar \chi \wedge f) =0$; 
\item[ii.] $\dbar^{\ast} N_{p,q+1}$ is a bounded map;  
\item[iii.] the map $S_{p,q}: K_{(p,q)}^2(\D) \rightarrow W_{(p,q+1)}^{-1}(\D)$ defined by 
$S_{p,q}f= N_{p,q+1}(\dbar \chi \wedge f)$ is compact. 
\end{itemize}
Thus, there exists $C>0$ such that 
\begin{align}\label{combine3} 
\dfrac{1}{\ep}\left\|\left(|\nabla \chi| \right) 
N_{p,q+1}(\dbar \chi \wedge f)\right\|_{L_{(p,q+1)}^2(\D)}^2 
\leq& C \ep\|f\|^2_{L_{(p,q)}^2(\D)} + \frac{C_{\ep}}{\ep}\|S_{p,q}f\|_{-1}^2 
\end{align}
Combining \eqref{combine1}-\eqref{combine3}, for $f \in K_{(p,q)}^2(\D)$ we have 
\[\left\|H_{\chi}^{p,q}f\right\|_{L_{(p,q)}^2(\D)}^2 \leq (C+1)\ep\|f\|^2_{L_{(p,q)}^2(\D)} 
+ \frac{C_{\ep}}{\ep}\left\|S_{p,q}f\right\|_{-1}^2.\] 
By compactness estimates in Lemma \ref{LemCompEst},  we conclude that $H_{\chi}^{p,q}$ 
is compact for any $\chi\in  C_{0}^{\infty}(\D)$.

We finish the proof as follows. The symbol $\phi$ can be approximated uniformly 
on $\Dc$ by a sequence $\{\chi_j\} \subset C_{0}^{\infty}(\D)$ because $\phi\in C(\Dc)$ 
and $\phi=0$ on $b\D$. As shown above, $H_{\chi_j}^{p,q}$ is compact for all $j$. Then 
$H_{\phi}^{p,q}$ is the limit of $\{H_{\chi_j}^{p,q}\}$ in the operator norm.
Since compactness is preserved under convergence in operator norm, 
$H_{\phi}^{p,q}$ is compact on $K_{(p,q)}^2(\D)$.
\end{proof}

Now, we prove Theorem \ref{ThmForms}.

\begin{proof}[Proof of Theorem \ref{ThmForms}]
Since $T_{\phi}^{p,q}=0$ whenever $\phi=0$ we just need to prove 
the converse. So we assume that $\phi\not\equiv 0$. Using dilation and translation 
if necessary, we assume that $\overline{\mathbb{D}^n}\subset \D$ and 
$\delta_1\leq |\phi|^2\leq \delta_2$ on  $\overline{\mathbb{D}^n}$ for some 
$0<\delta_1\leq \delta_2<\infty$.  Let us denote $\phi_0(\xi)=\phi(\xi,0,\ldots,0)$ 
for $\xi\in \mathbb{D}$. 

Next we will use the following fact:  the multiplication 
operator $M_{\phi_0}$ is compact on $L^2(\mathbb{D})$ if and 
only if $\phi_0= 0$ on $\mathbb{D}$ 
(see \cite[Corollary 1.1]{SinghKumar1979}). Then the fact that 
$\phi_0\neq 0$ on $\mathbb{D}$ implies that there exists 
a sequence $\{f_j\}\subset L^2(\mathbb{D})$ and $\ep>0$ such that 
$\|f_j\|_{L^2(\disk)}=1$ for all $j$ and  $\|\phi_0(f_j-f_k) \|_{L^2(\disk)}^2>\ep $ 
for $j\neq k$. Now we extend each $f_j$ trivially as zero outside $\mathbb{D}$ and define 
\[ F_j(z)=f_j(z_1) dz_1\wedge \cdots \wedge dz_p\wedge d\zb_1\wedge \cdots \wedge d\zb_q.\]
Hence $\{F_j\}$ is a bounded sequence in $K^2_{(p,q)}(\D)$.  

Let $\psi\in C^{\infty}_0(\D)$ such that $\psi=1$ on  $\disk^n$. Then for $j\neq k$ we have 
\begin{align*}
\ep< \|\phi_0(f_j-f_k) \|_{L^2(\disk)}^2 
\leq&   \dfrac{\delta_2}{\delta_1\pi^{n-1}} 
\int_{\disk^n}|\phi(z)\psi(z)|^2|f_j(z_1)-f_k(z_1)|^2dV(z) \\
\leq& \dfrac{\delta_2}{\delta_1\pi^{n-1}} \| M_{\phi\psi}(F_j- F_k)\|^2_{L_{(p,q)}^2(\D)} 
\end{align*}
where $M_{\phi\psi}:L_{(p,q)}^2(\D)\to L_{(p,q)}^2(\D)$ is the multiplication operator 
by $\phi\psi$. Hence  $M_{\phi\psi}$ is not compact on $K_{(p,q)}^2(\D)$ as 
$\{M_{\phi\psi}F_j\} $ has no convergent  subsequence.

Since $\phi\psi\in C(\Dc)$ such that $\phi\psi =0$ on $b\D$,  Lemma \ref{compactnesshankel} 
shows that $H_{\phi\psi}^{p,q}$ is compact on $K_{(p,q)}^2(\D)$. Then by \eqref{multiplication}, 
the operator $T_{\phi\psi}^{p,q}$ is not compact on $K_{(p,q)}^2(\D)$. Therefore,  
we conclude that $T_{\phi}^{p,q}$ is not compact. 
\end{proof}

\section*{Acknowledgment}

We would like to thank Timothy Clos for reading an earlier manuscript of this paper  
and Trieu Le for discussions that improved some of the proofs.  We thank Phillip 
Harrington for pointing out \cite[Proposition 3.1]{Harrington2022} and the 
referee for drawing our attention to \cite[Corollary 1.1]{SinghKumar1979} 
which shortened the proof of Theorem \ref{ThmForms}. Finally, we thank the 
editor Harold Boas for helpful feedback.

\end{document}